\documentclass[reqno]{amsart}

\usepackage{xcolor}
\usepackage{esint}
\usepackage{graphicx}
\usepackage{caption}
\usepackage{subcaption}

\def\eps{\varepsilon}

\def\e{{\rm e}}

\def\d{{\rm d}}
\def \R{\mathbb{R}}
\def \C{\mathbb{C}}
\def \CC{\mathcal{C}}

\def\T{ \mathbb{T}}
\def\A{\mathcal{A}}
\def\eps{\varepsilon}

\providecommand{\kwds}[1]
{
  \small	
  \textbf{\textit{Keywords---}} #1
}

\newtheorem{theorem}{Theorem}

\theoremstyle{remark}
\newtheorem{remark}[theorem]{Remark}

\begin{document}
%
%
%
\title[Energy asymptotics]%
{Energy asymptotics for the strongly damped Klein-Gordon equation}

\author{Haidar Mohamad}
\email{Alexandre.Monnet@ku.de}
\address{Katholische Universit\"{a}t Eichst\"{a}tt-Ingolstadt\\
Auf der Schanz 49\\
85049 Ingolstadt\\
Germany}
 
\maketitle
\begin{abstract}
We consider the strongly damped Klein Gordon equation for defocusing nonlinearity and we study the asymptotic 
behaviour of the energy for periodic solutions.  We prove  first the exponential decay to zero for zero mean solutions. 
Then,  we characterize the limit of the energy,  when the time tends to infinity,  for  solutions with small enough initial data and we finally 
prove that such limit is not necessary zero. 
\end{abstract}
\kwds{Klein-Gordon equation,  Strong and weak damping,   Energy decay.}
 \section{Introduction}
    We consider the damped  nonlinear Klein-Gordon equation 
    \begin{equation}\label{MainEq}
    \partial_t^2 \psi +L  \partial_t\psi - \Delta \psi + \psi + |\psi|^p \psi =0 ,  \,\, \psi(0) = \psi_0,\,\, \partial_t\psi(0) = \psi_1,
    \end{equation}
    where $\psi:\CC( \R^+, X),$  $X$ is a Banach subspace in $L^2(\Omega\subset \R^d )$ and  $L: D(L)\subset L^2(\Omega) \to L^2(\Omega)$ is some non-negative operator. Damped semi-linear wave equations particularly equation \eqref{MainEq} have gained 
    a lot of attention numerically and analytically. The related literature is extensive. See for example \cite{HaPark:2011:GblExUnifDecKG, Xu:2010:GblExBlwUpKG, XuDing:2013:GblSolFtTimeKG, GaoGuo:2004:TimePer2DKG,   LinCui:2008:NewMthdKG, DiazPuri:2005:NumericDKG3}.  
  The associated energy $E \in \CC(X, \R^+) $ is given by 
  \begin{equation}
  E(\psi) = \frac 12 \int_{\Omega} |\partial_t\psi|^2 + |\psi|^2 + |\nabla \psi|^2 +\frac{1}{p+2} \int_{\Omega} |\psi|^{p+2}.
  \end{equation}
  The linear part of the system is dissipative in the sense that the linear semigroup action loses energy. Indeed, we have atleast  formally
  \begin{equation}\label{EnrDec}
  \frac{\d}{\d t} E(\psi(t)) =- \int_{\Omega} |\sqrt{L}\,\partial_t\psi|^2 \leq 0.
  \end{equation}
  In general,  damping can be {\it weak} when the semigroup generated by the linear part of the equation is merely continuous (example $L = \gamma(x) I$ with $\gamma(x) \geq 0$ ), 
  and {\it strong} when the semigroup is compact (example $L = -\Delta + \gamma(x) I$ with $\gamma(x) \geq 0$ ).  
  It has been proven that when $L = \gamma(x)$ is a "positive  multiplicator", namely, 
  \begin{equation}
  \forall  x \in \omega \subset \Omega,  \gamma(x) \geq \alpha > 0 
  \end{equation}
  for some open set $\omega$ and  a positive $\alpha,$ the energy $E$ decays exponentially to zero.
  We refer to \cite{AlouiIbrahimNakanishi:2010:ExpDecayDKG1,  DehmanLebeauZuazua:2003:StabCrlSWE,
  JolyLaurent:2013:StabSWE} for related results.   Another related results  of polynomially and exponentially decays a for slightly different types of damping   are proved by J. Royer \cite{Royer:2018:EnDecayDKG2} for the linear Klein-Gordon equation  and by \cite{Nakao:2012:EnDecKGlocDisp} for the nonlinear equation.

 The energy asymptotics for the strongly damped equation  hasn't gained enough attention. 
 R.  Xu and W.  Lian \cite{LianXu:2020:WeakStrongLogarithmicNL}  have recently considered the following  equation 
 $$ \partial_t^2 \psi - \omega \Delta \partial_t\psi  + \mu \psi - \Delta \psi + \psi + f(\psi) =0,  \omega \geq 0$$
 with logarithmic nonlinearity  given by $f(\psi) = -\ln(|\psi|)\psi$  and they proved,  under the strong assumption $\mu > \omega \lambda_1,$ that the related energy 
 decays exponentially to zero for three  different initial energy levels determined by the minima of the potential energy on the so-called  Nehari manifold
 \cite{Nehari:1960:NonLin2ndODE, Nehari:1960:CharVal2ndDE}.    $\lambda_1$ is the first eigenvalue of the operator$-\Delta$ under the homogeneous Dirichlet boundary conditions.  
 Considering a functional setting similar  to that used in \cite{LianXu:2020:WeakStrongLogarithmicNL} ,  S.M.S.  Cordeiro {\em et  al.} \cite{CordeiroPereira:2021:ExpDecKGKirchhoff} proved the exponential decay to zero  of the energy related to the 
strongly damped ($L = -\Delta$) Klein-Gordon equation of Kirchhoff-Carrier  type 
 $$\partial_t^2 \psi -\Delta \partial_t\psi - M\left( \| \nabla \psi \|\right)\Delta \psi +M_1\left( \|  \psi \|\right)\ \psi -\ln( |\psi|^2) \psi =0 ,$$
where $M, M_1$ are two continuous  non-negative  functions defined on $[0, +\infty).$
However,  depending on the initial data,  the energy  in the strongly damped case ($L = -\Delta$) might
 decay to some conserved  quantity different from zero.  For example,   the initially non-zero-average periodic solutions defined on the $d-$torus $\Omega =\T^d$ have such quantity.  This observation has not been studied so far and will be the focus of this paper.
 To clarify our  purpose,  let's consider first the linear equation defined by the linear part of \eqref{MainEq} and denote 
 $$\theta(t) := \fint  \psi(t) := \frac{1}{(2 \pi)^d}\int_{\T^d} \psi(t, x)\d x$$
 which corresponds to the time-dependent zero Fourier coefficient of $\psi(t, \cdot).$
 Thus, the zero average function $\phi := \psi - \theta$ satisfies the linear system 
 \begin{equation}
  \partial_t^2 \phi-\Delta \partial_t\phi - \Delta \phi + \phi  =0,
 \end{equation}
 which retrieves the exponential  decay  of $E(\phi)$ to zero.
  Moreover,  $E(\psi)$ decays to a conserved positive quantity; more specifically 
  
  $$E(\psi) = E(\phi) + \frac{(2 \pi)^d}{2}( |\theta|^2 + |\theta'|^2)$$
  and 
  $$ |\theta(t)|^2 + |\theta'(t)|^2= |\theta(0)|^2 + |\theta'(0)|^2.$$
  Indeed, $\theta$ satisfies the differential equation 
  $$\theta'' + \theta = 0.$$
 In the spectral level,  the damping effect of  $L= -\Delta$  is caused by  its  non-zero eigenvalues,  whereas $\theta$ oscillates independently, so that quantity $ |\theta(t)|^2 + |\theta'(t)|^2$ remains conserved.  This separation is however not clear in the case of 
 the nonlinear equation.   The rest of the paper is organised as follows: We prove first that the Cauchy problem of \eqref{MainEq} ($L = -\Delta$) is globally well posed 
 on $H^2(\T^d).$ Then,  we  prove that the  energy of zero-average solutions decays exponentially to zero.   Finally,  we study the energy asymptotics for small initial data solutions. 
 \subsection*{Useful notations}
 \begin{itemize}
 \item The norm $\|\cdot\|$ refers to the $L^2$ norm. 
 \item The constant $C$ changes in the estimates from line to line unless otherwise noted.
 \end{itemize}

\section{The Cauchy problem}
In this section we study the Cauchy problem for \eqref{MainEq} on $\CC(\R^+,  H^2(\T^d)).$ 

We have 
\begin{theorem}\label{CauchyPbThm}
Assume that  $d\leq 3,$  $p\geq 0$ and  $\psi_0, \psi_1 \in H^2(\T^d).$
Equation \eqref{MainEq} has a unique global solution $\psi \in\CC^1(\R^+,  H^2(\T^d))$  such that $\psi(0) = \psi_0$ 
and $\partial_t\psi(0) = \psi_1.$  Moreover,  there exists a constant $C >0$
such that 
\begin{equation}\label{H2bound}
\|\psi(t)\|_{H^2} \leq C\left(\|\psi_0\|_{H^2} + \|\psi_0\|^{\frac p2 +1}_{p+2} + \|\psi_1\|  \right) \,\, \forall t \in \R^+. 
\end{equation} 
\end{theorem}
\begin{proof}
Written as a first order system,  equation \eqref{MainEq}  takes the abstract form 
\begin{equation}\label{1stOrdMainEq}
\partial_t \Psi + \mathcal{A} \Psi + F(\Psi)=0
\end{equation}
with
$$ \Psi  := \left( \begin{matrix} \psi \\ \partial_t \psi  \end{matrix}\right),\,\, 
\mathcal{A} := \left( \begin{matrix} 0 & -1\\ -\Delta +1 & - \Delta \end{matrix}\right),\,\,  
\text{and}\,\,  F(\Psi) :=   \left( \begin{matrix} 0 \\  |\psi|^p \psi\end{matrix}\right).$$
We cast \eqref{1stOrdMainEq} in its mild formulation 
\begin{equation}\label{MildFormMainEq}
\Psi(t) = \e^{- t\A} \Psi(0) - \int_0^t \e^{(\tau - t) \A} F(\Psi(\tau)) \d \tau,
\end{equation}
where 
$$\e^{- t\A} = \e^{\frac 12 t \Delta} \left( \begin{matrix} \cosh(t A) - \frac 12 \Delta A^{-1} \sinh(t A) & A^{-1}  \sinh(t A) \\  (\Delta -1) A^{-1}  \sinh(t A) &\cosh(t A) + \frac 12  \Delta A^{-1} \sinh(t A) \end{matrix}\right)$$
with 
$$A:= \frac 12 \sqrt{ \Delta^2 + 4 \Delta -4 } .$$
Taking into account that $d\leq 3,$  it is classical that the map $$\Psi \mapsto F(\Psi)$$ leaves $X_2 := H^2\times H^2$ invariant,
and is Lipschitz continuous on the bounded subsets of  $X_2.$  Consequently,  \eqref{MildFormMainEq} has a maximal solution 
$\Psi \in \CC([0,  T^\ast),  X_2),$  which blows up as $t$ approaches $T^\ast$ if $T^\ast <\infty.$ We prove now that $T^\ast = \infty.$
Denote $$J(\psi):= \frac 12\|\Delta \psi - \partial_t \psi\|^2 + \frac 12 \|\psi\|^2_{H^1} + \frac{1}{p+2}\|\psi\|_{p+2}^{p+2}.$$
It is clear that $$J(\psi) = E(\psi) +  \frac 12\|\Delta \psi\|^2 - {\rm Re}\int_{\T^d} \partial_t \psi \Delta \bar \psi.$$
Differentiating $J$ with respect to time for $t<T^\ast$ and using \eqref{EnrDec},  we find that 
\begin{eqnarray}
\frac{\d}{\d t} J(\psi) &=& \frac{\d}{\d t} E(\psi) +{\rm Re}\int_{\T^d} \partial_t \Delta \psi \Delta \bar \psi 
- \partial^2_t \psi \Delta \bar \psi - \partial_t \psi \Delta  \partial_t\bar \psi\notag \\
&=& {\rm Re}\int_{\T^d} (-\Delta\psi + \psi + |\psi|^p\psi )  \Delta \bar \psi\notag \\
&=& - \|\Delta \psi\|^2 -  \|\nabla \psi\|^2- \int_{\T^d} |\psi|^p  |\nabla \psi|^2 +p |\psi|^{p-2}  |{\rm Re} (\psi\nabla \bar \psi)|^2\notag
\end{eqnarray}
which means that $t\mapsto J(\psi(t))$ is decreasing and  we have 
$$J(\psi(t)) \leq J(\psi(0))  \,\, \forall t\in [0 , T^\ast).$$
Consequently,  there exists a constant $C>0$ depending on $\psi_0$ and $\psi_1$ which can be expressed in the form of the r.h.  side of \eqref{H2bound} such that 
\begin{equation}\label{psiBndH2}
\|\psi(t)\|_{H^2} \leq C \,\, \forall t\in [0 , T^\ast).
\end{equation}
Using \eqref{MildFormMainEq} together with \eqref{psiBndH2},  we find that there exists a   constant $C_1>0$ depending on $\psi_0$ and $\psi_1$ such that for any $0< T<T^\ast,$ we have 
$$\|\Psi(t)\|_{L^\infty([0, T],X_2)} \leq  \|\Psi(0)\|_{X_2} + C_1 T,$$
which implies that $\Psi$ can be extended as a global solution in $\CC(\R^+,  X_2).$ 
\end{proof}

\begin{remark}\label{RemH1Sol}
Giving more restrictions on the nonlinearity,  a global well posedness result for  the Cauchy problem of \eqref{MainEq} could be established on a 
larger spacial domain.  If, for example,  $\psi_0, \psi_1 \in H^1(\T^d)$ and 
\[
0<  p  
\begin{cases}
\leq \frac{4}{d-2}  & \text{if $d\geq 3$,} \\
 < \infty & \text{if $d = 1,  2,$} 
\end{cases}
\]
then \eqref{MainEq}  has a  unique global  solution 
$$\psi \in \CC(\R^+ , H^1(\T^d))\cap \CC^1(\R^+, L^2(\T^d))\cap \CC^2(\R^+ , H^{-1}(\T^d))$$
with $\partial_t \psi \in L^2_{{\rm loc}}(\R^+,  H^1(\T^d)).$   The proof for such result follows the same steps of Theorem 3.1 in  \cite{GazzolaSquassina:2006:GlobalSolDmpdSWE} 
for the the existence of a unique maximal solution. The fact that the energy is decreasing implies the boundness  of  $t\mapsto \|\psi(t)\|_{H^1}$ which in terns imply that the solution is global. 
However,  the spacial regularity given by Theorem \ref{CauchyPbThm} is needed  in the main result presented in Theorem \ref{ThmDec}.
\end{remark}
\section{Asymptotic behaviour of energy}

 We consider now the energy decay for zero mean solutions.
\begin{theorem}
Let $\psi$ denote a solution of \eqref{MainEq} in the same functional settings mentioned in Remark \ref{RemH1Sol}.  Assume further that 
\begin{equation}
  \fint \psi(t) = 0, \,\, \forall t \in \R^+.
\end{equation}
Then,  there exists $C,  \alpha >0$ such that 
\begin{equation}
E(\psi(t)) \leq C {\rm e}^{-\alpha t}  , \,\, \forall t \in \R^+.
\end{equation}
\end{theorem}
\begin{proof}
The proof uses the same technique used to prove Proposition 2.6 in  \cite{JolyLaurent:2013:StabSWE} together with the Poincar\'{e}
inequality 
$$\|\partial_t\psi \|=\left\|\partial_t\left(\psi - \fint \psi\right)\right\| \leq C \|\nabla \partial_t \psi\|$$
for some constant $C>0.$ We introduce the modified energy
$$E_\eps(\psi) : = E(\psi) +\eps  \int_{\T^d} {\rm Re}(\bar \psi\partial_t \psi )$$
with $\eps>0.$ For $\eps$ small enough, there exist two constants $C_1, C_2 >0$ depending on $\eps$ such that 
$$C_1(\|\psi\|^2 + \|\partial_t \psi\|^2) \leq  \frac 12(\|\psi\|^2 + \|\partial_t \psi\|^2) +\eps  \int_{\T^d} {\rm Re}(\bar \psi\partial_t \psi ) \leq C_2(\|\psi\|^2 + \|\partial_t \psi\|^2),$$
which means that $E_\eps$ is equivalent to $E$ for small enough $\eps$ and it is sufficient to prove the exponential decay for $E_\eps.$
We have 
\begin{eqnarray}
\frac{\d }{\d  t}E_\eps(\psi(t)) &=& \frac{\d }{\d  t}E(\psi(t)) + \eps \|\partial_t \psi\|^2 + \eps  \int_{\T^d} {\rm Re}(\bar \psi\partial^2_t \psi )\notag\\
&=& - \|\nabla \partial_t \psi\|^2 + \eps \|\partial_t \psi\|^2 -\eps \| \psi\|^2 -\eps \|\nabla \psi\|^2 \notag\\
&&  -\eps \|\psi\|_{p+2}^{p+2} -  \eps  \int_{\T^d} {\rm Re}(\nabla\bar \psi\nabla \partial_t \psi )\notag\\
&\leq& -(C^{-1}-\eps -\frac{\eps}{2} C^{-1}) \|\partial_t \psi\|^2-\eps \| \psi\|^2  -\frac{\eps}{2} \|\nabla \psi\|^2  -\eps \|\psi\|_{p+2}^{p+2}\notag\\
&\leq & - \beta E_\eps(\psi(t))\notag
\end{eqnarray}
with $\beta = \beta(\eps) >0$ which implies the exponential decay of $E_\eps$ and thus of $E.$
\end{proof} 

We study now the energy decay for solutions with small initial data.  Namely, 

\begin{theorem}\label{ThmDec}
Let $\psi_0,  \psi_1 \in H^2(\T^d).$ 
Denote $\psi$ the solution of  \eqref{MainEq} given by Theorem  \ref{CauchyPbThm} such that 
$$\psi(0) = \psi_0,\, \,  \partial_t\psi(0)=  \psi_1.$$
 Denote further 
 $$\theta := \fint\psi, \,\,  \phi := \psi - \theta\,\,\, \text{and}$$
 $$ Q(\theta) := (2 \pi)^d\, \left(\frac 12 |\theta|^2 + \frac 12 |\theta'|^2 + \frac{1}{p+2}|\theta|^{p+2} \right).$$
 Then,  if $\|\psi_0\|_{H^2}, \|\psi_1\|_{H^2}$ are small enough, there exist $\tilde C,  \beta, \alpha> 0$  depending on  $\psi_0$ and $\psi_1$ 
 and  there exists $C>0$  such that 
 \begin{equation}\label{H2Decay}
 \|\phi\|_{H^2} + \|\partial_t\phi\|_{H^2} \leq C ( \|\phi(0)\|_{H^2} + \|\partial_t\phi(0)\|_{H^2})\e^{-\beta t}.
 \end{equation}
 Moreover, the limit $\lim_{t\to \infty}Q(\theta(t))$ exists and we have
 \begin{equation}\label{ElimQDecay}
 |E(\psi(t)) - Q(\theta(t))| \leq \tilde C\, {\rm e}^{-\alpha\, t}.
 \end{equation}
\end{theorem}

\begin{proof}
Denote $\phi = \psi - \theta$ and $f(z) = |z|^p z.$ Thus,  $(\phi, \theta)$ satisfies  the system 
\begin{subequations}
\begin{gather}
 \partial_t^2 \phi - \Delta \partial_t\phi - \Delta \phi + \phi + f(\psi) -  \fint f(\psi) =0  \,,
 \label{MainEqphi.a} \\
 \theta '' + \theta  + \fint f(\psi)= 0 \label{MainEqphi.b} \,.
\end{gather}
\end{subequations}
Denote 
$$ \Phi := \left( \begin{matrix} \phi\\\partial_t \phi    \end{matrix}\right)\,\, \text{and}\,\,
G(\psi) :=  \left( \begin{matrix} 0 \\  f(\psi) -  \fint f(\psi)    \end{matrix}\right).$$
Thus,  as in \eqref{MildFormMainEq},  $\Phi$ satisfies the mild equation 
\begin{equation}\label{MildFormPhi}
\Phi(t) = \e^{- t\A} \Phi(0) - \int_0^t \e^{(\tau - t) \A} G(\psi(\tau)) \d \tau.
\end{equation}
Using the Pioncar\'{e} inequality,  there is a constant $C>0$ such that 
\begin{equation}\label{EstimG}
\|G(\psi)\|_{H^2} \leq C \|\nabla f(\psi)\|_{H^1}.
\end{equation}
Moreover, we have 
$$|\nabla f(\psi)| \leq (p+1)|\psi|^p |\nabla \psi| = (p+1)|\psi|^p |\nabla \phi|, $$
and  
$$|\Delta f(\psi)| \leq (p+1)(|\psi|^p |\Delta \psi| +p|\psi|^{p-1} |\nabla \psi|^2)= 
(p+1)(|\psi|^p |\Delta \phi| +p|\psi|^{p-1} |\nabla \phi|^2).$$
Thus,  using the interpolation inequality 
$$\||\nabla \psi|^2\| \leq \sqrt 2\|\nabla \psi\|  \|\Delta \psi\|  $$
together with the Sobolev inequality 
$$\|\psi\|_{L^\infty} \leq C \|\psi\|_{H^2},$$
there exists a constant $C>0$ such that 
\begin{equation}\label{EstimG1}
 \|\nabla f(\psi)\|_{H^1} \leq C  \|\psi\|^p_{H^2}    \|\phi\|_{H^2}.
\end{equation}
For any $\varphi \in H^2(\T^d)$ with $\fint \varphi  = 0,$ we have 
$$\|\e^{\frac 12 t \Delta} \sinh(t A) \varphi\|_{H^2} \leq \e^{-\frac 12 t} \| \varphi\|_{H^2}, \,\,\,  
\|\e^{\frac 12 t \Delta} \cosh(t A) \varphi\|_{H^2} \leq \e^{-\frac 12 t} \| \varphi\|_{H^2},$$
which implies that,  for any $\Upsilon \in X_2$ with $\fint \Upsilon = 0,$ we have 
\begin{equation}\label{DecaySemiGp}
\left\|\e^{-t \A} \Upsilon \right\|_{X_2} \leq C  \e^{-\frac 12 t} \| \Upsilon\|_{X_2}
\end{equation}
for some $C>0.$
Then,  combining \eqref{MildFormPhi},  \eqref{EstimG},  \eqref{EstimG1} and \eqref{DecaySemiGp} together with estimate 
\eqref{H2bound},    we get 
\begin{eqnarray*}
\|\Phi(t)\|_{X_2}& \leq & C_1   \e^{-\frac 12 t} \|\Phi(0)\|_{X^2} + C_2 \int_0^t   \e^{-\frac 12 (t-\tau)}    \|\psi(\tau)\|^p_{H^2}\|\Phi(\tau)\|_{X_2}\d \tau\\
&\leq &   C_1   \e^{-\frac 12 t} \|\Phi(0)\|_{X_2} + C_2 C(\psi_0, \psi_1) \int_0^t   \e^{-\frac 12 (t-\tau)} \|\Phi(\tau)\|_{X_2}\d \tau,\\
\end{eqnarray*}
for some constants $C_1, C_2 >0.$ Since $H^2(\T^d)$ is embedded continuously in $L^{p+2}(\T^d)$ and using \eqref{H2bound},  we find that for small enough $\|\psi_0\|_{H^2}, \|\psi_1\|,$  we have 
$$\beta: = \frac12 - C_2 C(\psi_0, \psi_1) >0,$$
which implies, using Gronwall inequality,  
\begin{equation}\label{H2Decay}
\|\Phi(t)\|_{X^2}  \leq  C_1   \|\Phi(0)\|_{X^2}  \e^{-\beta t}.
\end{equation} 
Since 
$$\left| |\phi + \theta|^{p+2}-|\theta|^{p+2}\right| \leq (p+2)2^p |\phi| (|\phi|^{p+1} + |\theta|^{p+1})$$
and using the continuous  embedding  $H^2(\T^d)\subset L^{p+2}(\T^d) $ together with \eqref{H2Decay},  
we get  
\begin{eqnarray}\label{EQtEstim}
|E(\psi(t)) - Q(t)| &\leq & E(\phi(t)) + \frac{1}{p+2} \left| \|\psi(t)\|_{p+2}^{p+2} - (2\pi)^d |\theta(t)|^{p+2} \right|\notag\\
&\leq& C \e^{-\alpha t}
\end{eqnarray}
for some constants $C,  \alpha >0$ depending on $\|\psi_0\|_{H^2}$ and $\|\psi_1\|.$ Since $t\mapsto E(\psi(t))$ is decreasing (see \eqref{EnrDec}), 
 the limit $\lim_{t\to \infty}  Q(t)$ exists  and \eqref{ElimQDecay} holds.
\end{proof}
We prove in the following theorem that $E(\psi)$ doesn't decay necessarily to zero.
\begin{theorem}\label{ThmNon0Dec}
Assume that $\frac 32\leq p\leq 4$.  Keeping the other settings of Theorem \ref{ThmDec},  there exist $\psi_0,  \psi_1 \in H^2(\T^d)$ for which we have
$$\lim_{t\to \infty} Q(t) >0.$$
\end{theorem}
\begin{proof}
Writing $\psi = \phi + \theta$ and  using the Taylor series with integral remainder, we find that 
\begin{eqnarray}\label{TaylorNonLin}
 f(\psi)&=&  f(\theta) +  \frac{\d}{\d x }f(x \phi + \theta)|_{x = 0} 
 +\int_0^1 \frac{\d^2}{\d x^2 }f(x \phi + \theta)|  ( 1- x) \d  x\notag\\
  & =& |\theta|^p \theta +\frac \eps 2 |\theta|^{p-2}( (p+2)|\theta|^2\phi  +p\, \theta^2 \bar \phi)  \\
  & +& \frac{p(p+2)}{4} \int_0^1 | x \phi + \theta|^{p-2}  (2  (x \phi + \theta) |\phi|^2 +\overline{(x \phi + \theta)} \phi^2) (1 - x)\d \sigma \notag\\
  &+& \frac{p(p-2)}{4}  \int_0^1 | x \phi + \theta|^{p-4}  (x \phi + \theta)^3 \bar \phi^2 (1 - x)\d \sigma. \notag
\end{eqnarray}
Since $\frac 32\leq p\leq 4,$ for any $x, y, z \in \C,$ we have 
$$|z| |y|^{p-1} \leq C (|y| + |z|^2 +  |y|^{p+2})$$
and then 
 \begin{align*}
 |x|^2 |x+y|^{p-1} |z|  & \leq C(|x|^{p+1} + |x|^2) (|z| |y|^{p-1} +|z|) \\
 & \leq C (|x|^{p+1} + |x|^2) (|z| + |y| + |z|^2 +  |y|^{p+2})
  \end{align*}
  for some constant $C = C(p) >0.$
  Thus,  we can write 
  \begin{align*}
  |Q'| & =(2\pi)^d \left|{\rm{Re}}\left(\left(\fint f(\psi)-  |\theta|^p \theta\right) \bar{\theta'}\right)\right|\\
  &\leq  C \fint \int_0^1 |\phi|^2 |x\phi + \theta|^{p-1}|\theta'| (1-x) \d x\\
  & \leq C (|\theta ' | + |\theta | + |\theta '|^2 +  |\theta|^{p+2}) \fint( |\phi|^{p+1} + |\phi|^2)\\
  & \leq  C ( \|\phi \|_{p+1}^{p+1} + \|\phi \|^2) (Q  + \sqrt{Q}),
  \end{align*}
  where $C = C(p)>0.$  Denote now $S := \frac 12 Q$ and $\eta(t) : = \frac 12 C \int_0^t  ( \|\phi(\tau) \|_{p+1}^{p+1} + \|\phi(\tau) \|^2)\d \tau.$ Thus,  we have  
  $$|S'|\leq \frac 12 C ( \|\phi \|_{p+1}^{p+1} + \|\phi \|^2) (S + 1).$$
 A direct application of Gronwall lemma implies 
 $$S(t) \geq \e^{-  \eta(t)} S(0) 
- \frac12 C  \int_0^t(\|\phi(\tau) \|_{p+1}^{p+1} + \|\phi(\tau) \|^2)\e^{  (\eta(\tau)-\eta(t))}  \d \tau.$$
Using \eqref{H2Decay} together with the Sobolev embedding $H^2(\T^d)\subset L^{p+1}(\T^d),$ we find that  there exist 
$\Gamma = \Gamma(\|\phi(0)\|_{H^2},  \|\partial_t\phi(0)\|_{H^2} )>0$
and $\gamma = \gamma(\|\psi_0\|_{H^2},  \|\psi_1\|)>0$ such that 
$$ \eta(t) \leq \frac{\Gamma}{\gamma},\,\, \forall t\geq 0.$$
Moreover,   following the development of the constant $\beta$ in \eqref{H2Decay} and for fixed  
$(\theta(0), \theta'(0)) \neq (0,0), $ we have $$\lim_{(\|\phi(0)\|_{H^2},  \|\partial_t\phi(0)\|_{H^2} ) \to (0,0)}\frac{\Gamma}{\gamma} = 0. $$
Thus,  there exist   $\psi_0,  \psi_1 \in H^2$ and $\delta >0 $ such that 
$$S(0) > \left( \frac{\Gamma}{\gamma} +\delta \right) \e^{\frac{\Gamma}{\gamma}} \geq (\eta(t) + \delta)  \e^{  \eta(t)},$$
which implies that  $$\lim_{t\to\infty} S(t) \geq \delta > 0$$
and completes the proof.
\end{proof}

\section{Numerical tests}
In this section we study numerically the energy asymptotics for \eqref{MainEq} with $L = -\Delta,$ $p =2$ and $d = 1,  2.$ The undamped equation ($L = 0$) has a conserved 
quantity $E(\psi).$ Then,  for long time simulations, one wants to construct numerical methods that approximately conserve this energy .  When using Fourier spectral methods, we primarily need to ensure that the time discretization preserves these property,  since the spectral spatial discretization will typically automatically satisfy it.  We suggest therefore to use 
the time stepping discretization $$UD_n := \frac{\psi_{n+1}- 2 \psi_n + \psi_{n-1}}{(\delta t)^2} +  (I - \Delta)\frac{\psi_{n+1}+2 \psi_n + \psi_{n-1}}{4} + |\psi_n|^2\psi_n, $$ 
where $\psi_n$ is the approximation of  $\psi(n\, \delta t)$ and $\delta t$ is the time step.   Thus, the scheme is given by
\begin{equation}
UD_n - \Delta \frac{\psi_{n+1}-\psi_n}{\delta t}= 0. 
\end{equation}
For the approximation of $E(\psi),$ we consider the quantity 
$$E_n(\psi) := \frac 12 \int_{\T^d} \left|\frac{\psi_n-\psi_{n-1}}{\delta t}\right|^2 +  \left|\frac{\psi_n+\psi_{n-1}}{2}\right|^2 +  \left| \nabla \frac{\psi_n+\psi_{n-1}}{2}\right|^2
+ \frac 14 \int_{\T^d} \left|\frac{\psi_n+\psi_{n-1}}{2}\right|^4  $$
which is approximately conserved for the scheme 
$$UD_n=0$$
discretizing the undamped equation  \eqref{MainEq} with ($L= 0$).  We study the long time asymptotic behaviour  of 
$E_n(\psi),$  
$$ E_n(\phi): = \frac 12 \int_{\T^d} \left|\frac{\phi_n-\phi_{n-1}}{\delta t}\right|^2 +  \left|\frac{\phi_n+\phi_{n-1}}{2}\right|^2 +  \left| \nabla \frac{\phi_n+\phi_{n-1}}{2}\right|^2
+ \frac 14 \int_{\T^d} \left|\frac{\phi_n+\phi_{n-1}}{2}\right|^4$$
and $$Q_n(\theta)  :=  (2 \pi)^d \left(\frac 12 |\theta_n|^2 + \frac 12 \left|\frac{\theta_n-\theta_{n-1}}{\delta t}\right|^2 + \frac 14 |\theta_n|^4\right),$$
where $\theta_n$ is the zero coefficient of the discrete Fourier transform applied on $\psi_n$ and  $\phi_n := \psi_n - \theta_n.$

\begin{figure}

\begin{subfigure}{0.47\textwidth}
\includegraphics[width=\textwidth]{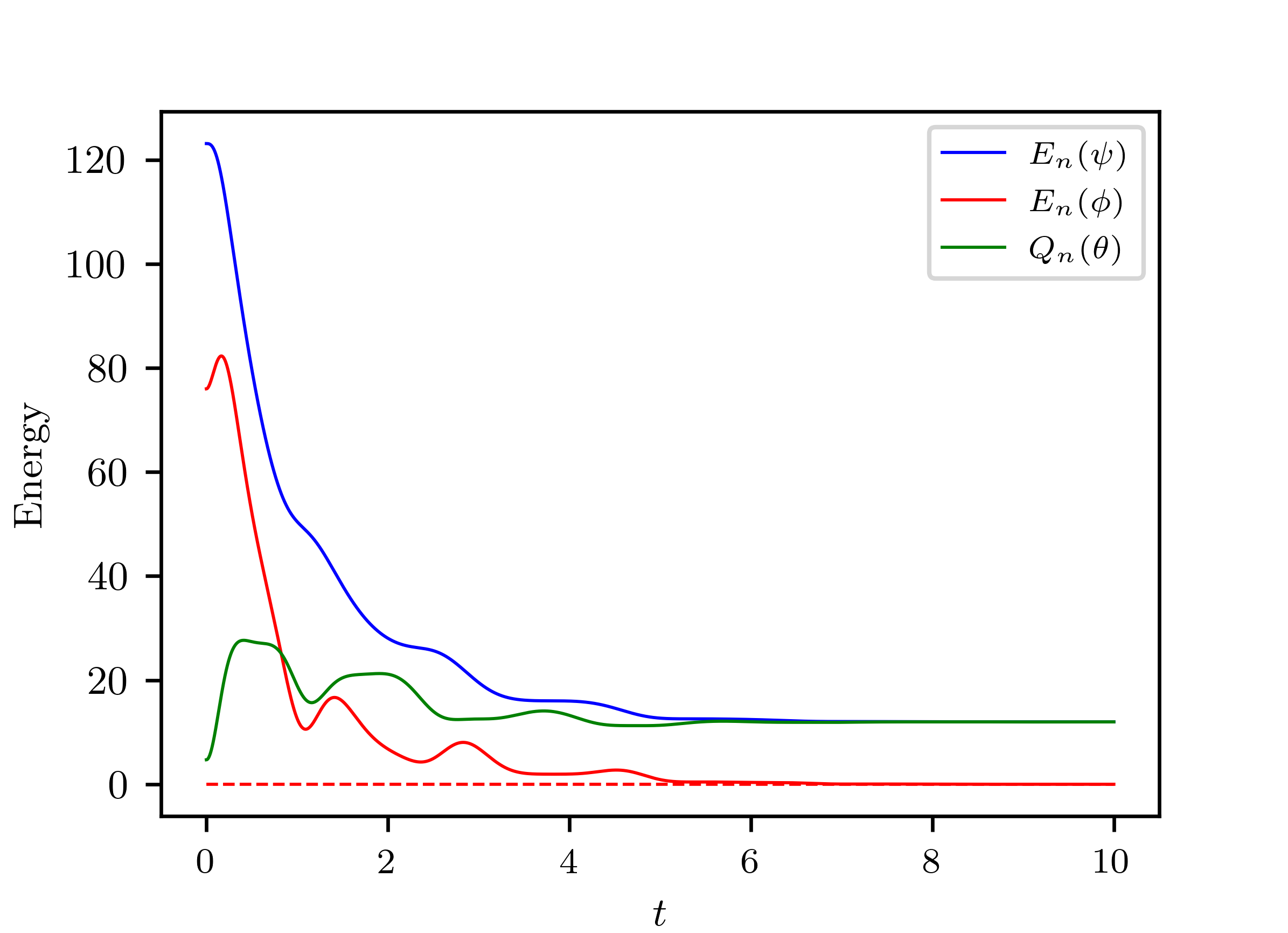} 
\label{subim1d1}
\end{subfigure}
\begin{subfigure}{0.47\textwidth}
\includegraphics[width=\textwidth]{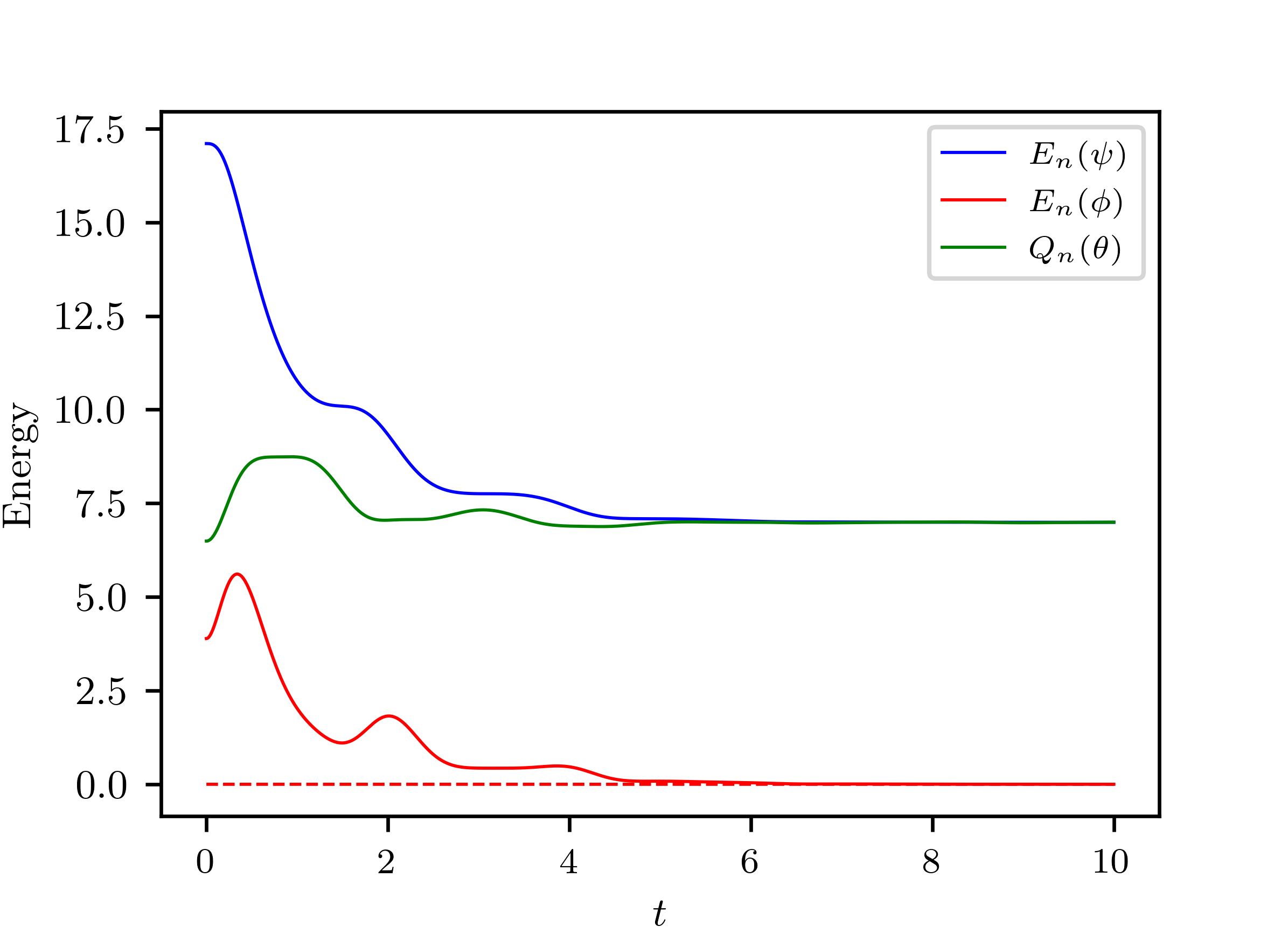}
\label{subim1d2}
\end{subfigure}

\caption{Approximate representation for long time asymptotics of $E(\psi),  E(\phi)$ and $Q$ by
$E_n(\psi),  E_n(\phi)$ and $Q_n$ respectively with $d=1,\,\,$ $\psi_1(x) = 0$ and left: $\psi_0(x) = 1+3\cos(x);$  right: 
$\psi_0(x) = (1+0.5\cos(x))^2.$  }
\label{imag1d}
\end{figure}

\begin{figure}

\begin{subfigure}{0.47\textwidth}
\includegraphics[width=\textwidth]{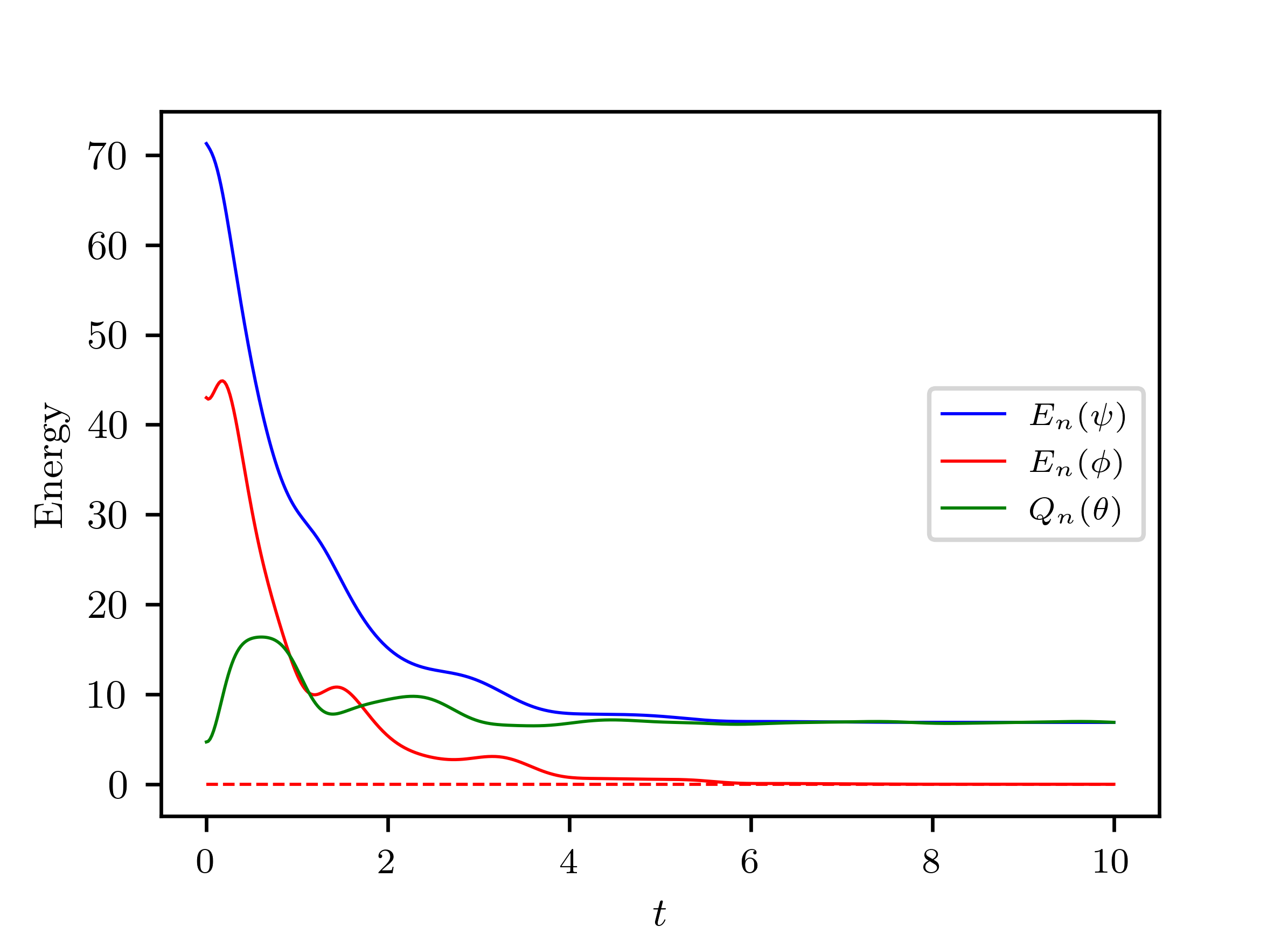} 
\label{subim2d1}
\end{subfigure}
\begin{subfigure}{0.47\textwidth}
\includegraphics[width=\textwidth]{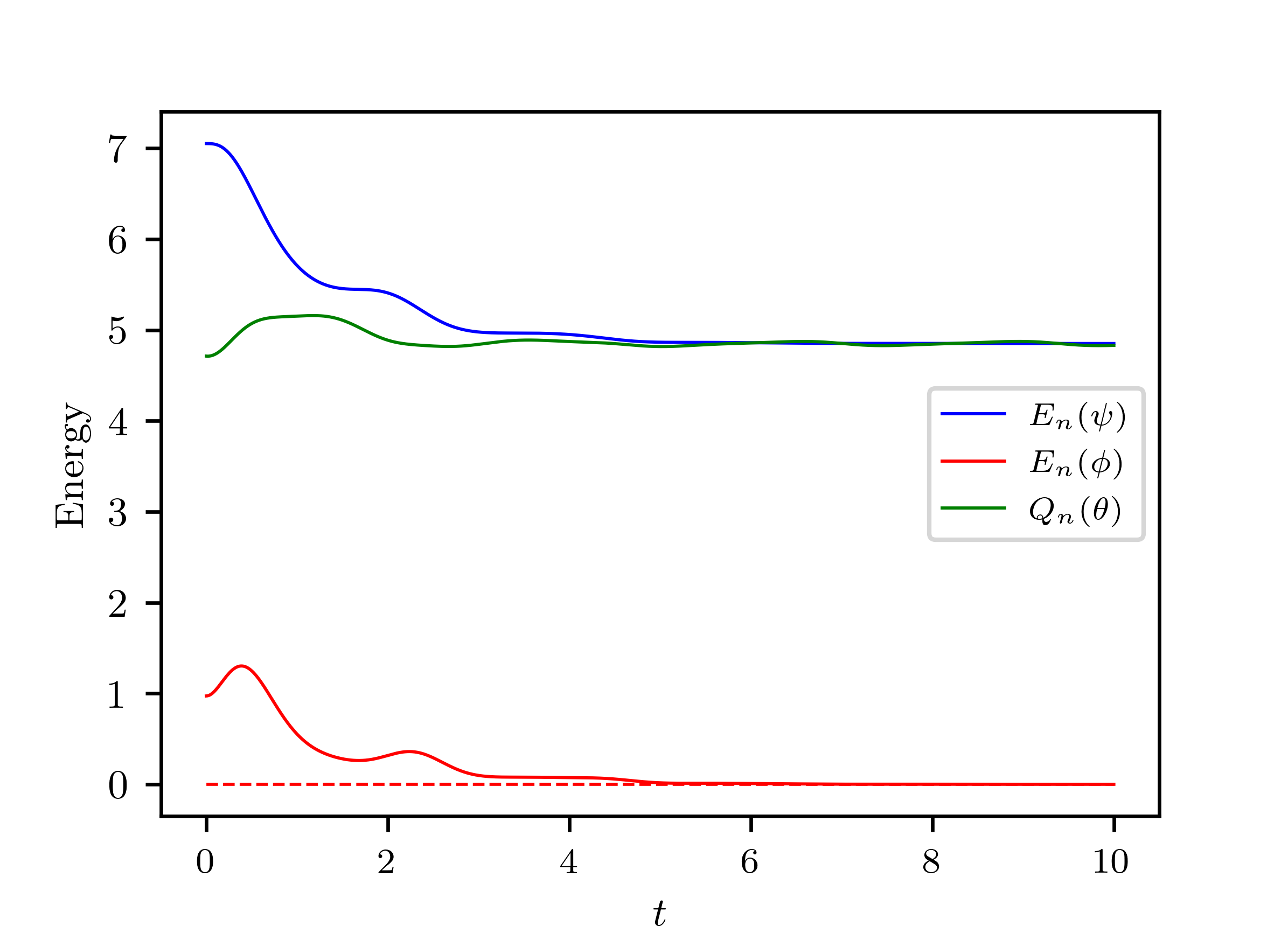}
\label{subim2d2}
\end{subfigure}

\caption{Approximate representation for long time asymptotics of $E(\psi),  E(\phi)$ and $Q$ by
$E_n(\psi),  E_n(\phi)$ and $Q_n$ respectively with $d=2,\,\,$  and left: $\psi_0(x, y) = 1+\cos(x) + 2 \cos(y),\,\,  \psi_1(x,  y) = \sin(x) + 2 \sin(y);$  right: 
$\psi_0(x, y) = 1+0.2\cos(x) + 0.5 \cos(y),\,\,  \psi_1(x, y) = 0.$  }
\label{imag2d}
\end{figure}

\begin{figure}

\begin{subfigure}{0.47\textwidth}
\includegraphics[width=\textwidth]{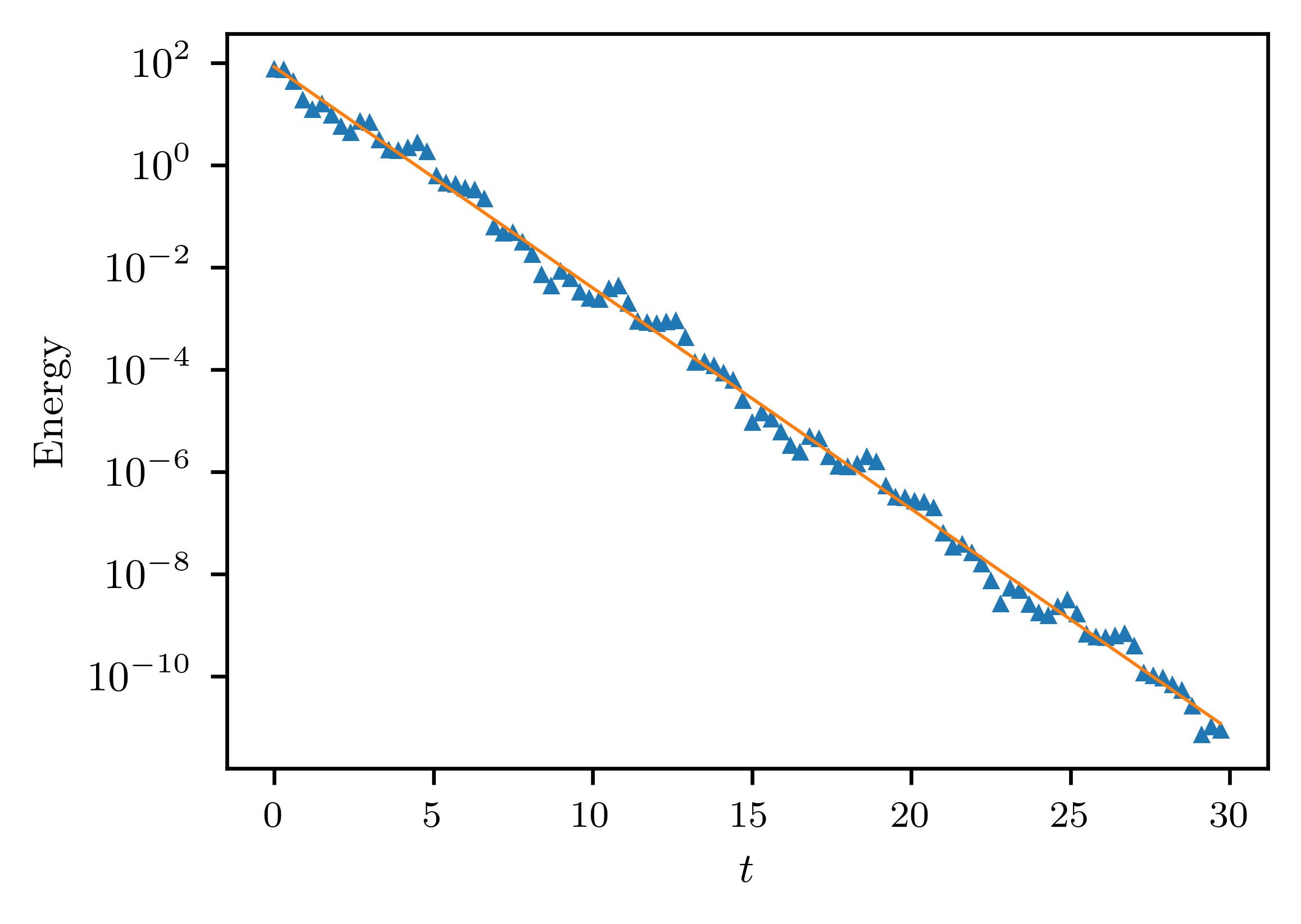} 
\label{subimEnExpDec1d1}
\end{subfigure}
\begin{subfigure}{0.47\textwidth}
\includegraphics[width=\textwidth]{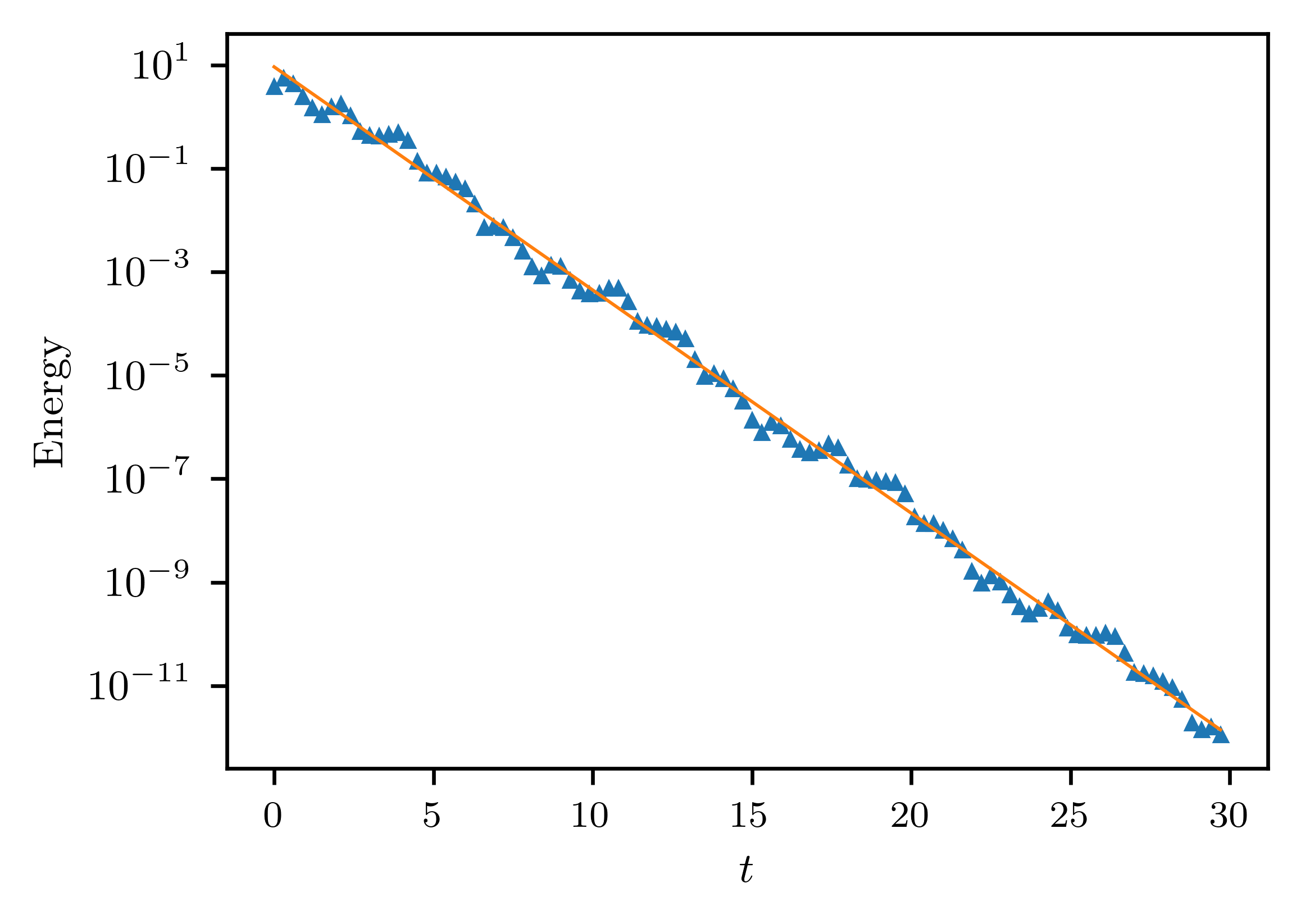}
\label{subimEnExpDec1d2}
\end{subfigure}

\caption{Time decay of $ E_n(\phi)$ when  $d=1,\,\,$ $\psi_1(x) = 0$ and left: $\psi_0(x) = 1+3\cos(x);$  right: 
$\psi_0(x) = (1+0.5\cos(x))^2.$  }
\label{imagEnExpDec1d}
\end{figure}

\begin{figure}

\begin{subfigure}{0.47\textwidth}
\includegraphics[width=\textwidth]{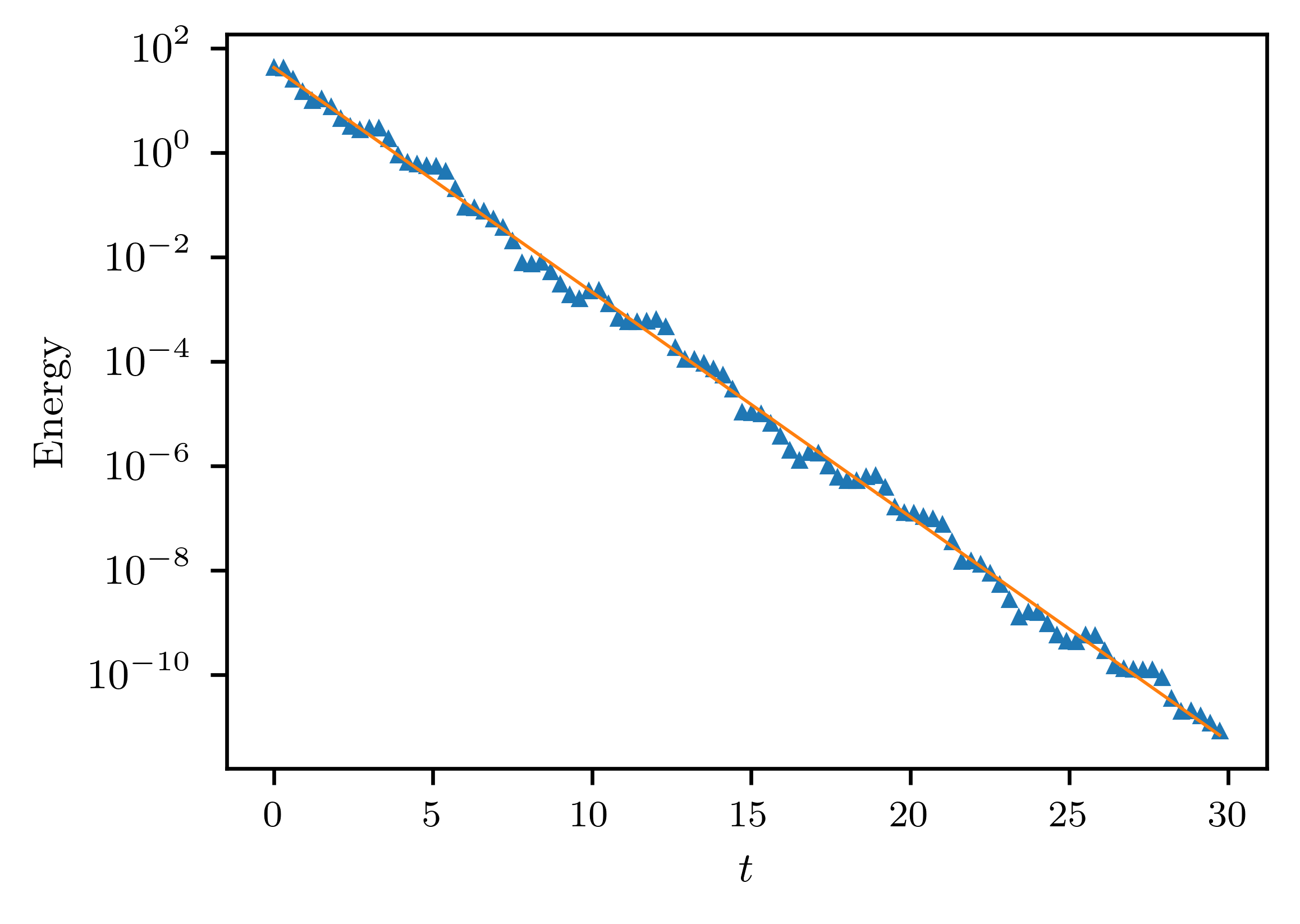} 
\label{subimEnExpDec2d1}
\end{subfigure}
\begin{subfigure}{0.47\textwidth}
\includegraphics[width=\textwidth]{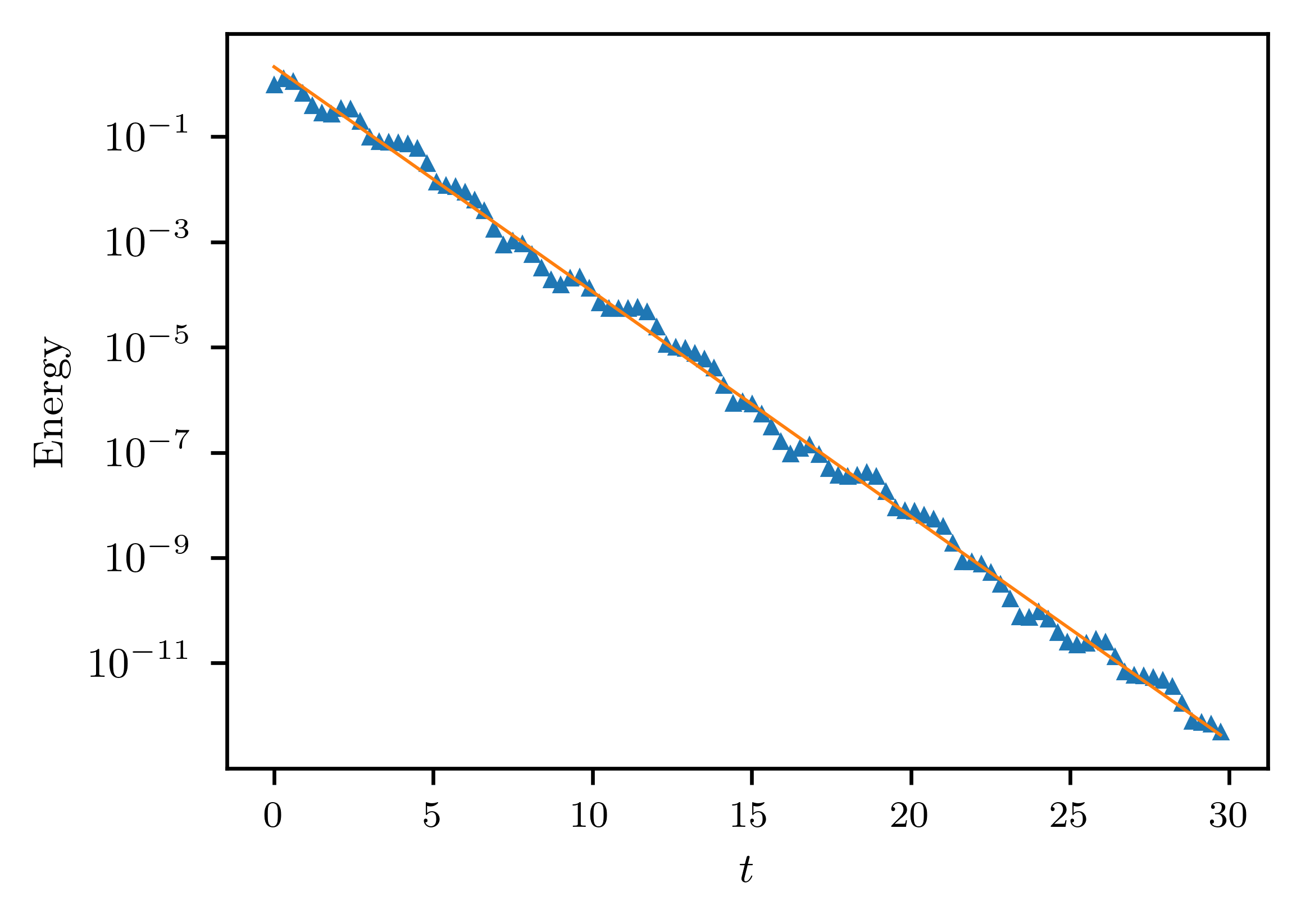}
\label{subimEnExpDec2d2}
\end{subfigure}

\caption{Time decay of $E_n(\phi)$ when $d=2,\,\,$  and left: $\psi_0(x, y) = 1+\cos(x) + 2 \cos(y),\,\,  \psi_1(x,  y) = \sin(x) + 2 \sin(y);$  right: 
$\psi_0(x, y) = 1+0.2\cos(x) + 0.5 \cos(y),\,\,  \psi_1(x, y) = 0.$  }
\label{imagEnExpDec2d}
\end{figure}
\newpage
\section{Discussion}
Our numerical results show,  for the solutions defined on the one and two dimensional spacial domains ($\T^d,\,\,  d = 1,  2$), that $E(\psi)$ and $Q(\theta)$ 
converge to  the same limit (Figures \ref{imag1d} and \ref{imag2d}) and that $E(\phi)$ decays  exponentially to zero  (Figures \ref{imagEnExpDec1d} and \ref{imagEnExpDec2d}) regardless of the assumptions of Theorem \ref{ThmDec} made on the initial data $\psi_0, \psi_1.$ In other words, 
we  believe that the   results of Theorem \ref{ThmDec}  are still valid for weaker assumptions on the initial data
and we leave the proof of such results as an open problem.  In general, the methods used in the literature (for example in \cite{AlouiIbrahimNakanishi:2010:ExpDecayDKG1,  DehmanLebeauZuazua:2003:StabCrlSWE,  JolyLaurent:2013:StabSWE}) to study the energy decay for the weakly damped Klein-Gordon equation having the damping operator $L = \gamma(x) I$ relay essentially on the fact that $L$ is bounded in the given functional setting which is not the case for the strongly damped equation with $L = -\Delta.$ Moreover,  in view of Theorem \ref{ThmNon0Dec},  the energy for the strongly damped equation doesn't decay necessarily to zero.   These two main differences make the 
methods used for the weakly damped equation unapplicable to study the energy asymptotics for the strongly damped one and lead to think differently to address the above open question.

\section*{Acknowledgments}
The author thanks Marcel Oliver for useful discussions.  The work was supported by the German Research Foundation grant
MO 4162/1-1.  The author further acknowledges support through German Research Foundations Collaborative Research Center TRR 181.  
 \bibliographystyle{siam}
\bibliography{kg.bib}
\end{document}